\documentclass[a4paper,12pt]{article}
\usepackage{amsmath,amsthm,amssymb,amsfonts}
\usepackage{bbm,bm}
\usepackage{latexsym}
\usepackage{mathrsfs}
\usepackage{threeparttable}
\usepackage{tabularx}
\usepackage{booktabs}
\usepackage{graphicx,subfigure}
\usepackage{color}
\usepackage{indentfirst}
\usepackage{geometry}
\usepackage{caption}
\usepackage{lineno}
\usepackage{abstract}
\usepackage{enumerate,mdwlist}
\usepackage[numbers,sort&compress]{natbib}
\usepackage[colorlinks,
            linkcolor=blue, 
            citecolor=red  
            ]{hyperref}
\usepackage{epstopdf}

\graphicspath{{figures/}}

\def \[{\begin{equation}}
\def \]{\end{equation}}

\newtheorem{thm}{Theorem}[section]

\newtheorem{lem}[thm]{Lemma}
\newtheorem{cor}[thm]{Corollary}

\newtheorem{remark}[thm]{Remark}

\begin{document}

\title{The complete forcing numbers of hexagonal systems\footnote{This work is supported by NSFC (Grant No. 11871256).}}
\author{ Xin He, Heping Zhang\footnote{Corresponding author.}\\
{\small School of Mathematics and Statistics, Lanzhou  University,  Lanzhou, Gansu 730000, P.R. China}\\
{\small E-mails: hex2015@lzu.edu.cn, zhanghp@lzu.edu.cn}}
\vspace{0.5mm}

\date{}

\maketitle

\noindent {\bf Abstract}: Let $G$ be a graph with a perfect matching. A complete forcing set of $G$ is a subset of edges of $G$ to which the restriction of every perfect matching is a forcing set of it. The complete forcing number of $G$ is the minimum  cardinality of complete forcing sets of $G$. Xu et al. gave a characterization for a complete forcing set and derived some explicit formulas for the complete forcing numbers of catacondensed hexagonal systems. In this paper, we consider general hexagonal systems. We present an upper bound on the complete forcing numbers of hexagonal systems in terms of  elementary edge-cut cover and two lower bounds by the number of hexagons and matching number respectively.  As applications, we obtain some explicit formulas for the complete forcing numbers of some types of hexagonal systems including parallelogram, regular hexagon- and rectangle-shaped hexagonal systems.

\vspace{2mm} \noindent{\it Keywords}: Hexagonal system; perfect matching; complete forcing set; complete forcing number

\vspace{2mm}


{\setcounter{section}{0}

\section{Introduction}
Let $G$ be a graph with vertex set $V(G)$ and edge set $E(G)$. A {\em matching}  of $G$ is a set of disjoint edges  of $G$. A {\em perfect matching} of $G$ is a matching that covers all vertices of $G$. A perfect matching of a graph coincides with a Kekul\'e structure of some molecular graph in organic chemistry.

Harary et al. \cite{Harary} applied the idea ``forcing'' to a perfect matching $M$ of $G$, which appeared in many research fields in graph theory and combinatorics \cite{Che,Mahmoodian}.  A {\em forcing set} of $M$ is a subset of $M$ contained in no other perfect matching of $G$. The minimum possible cardinality of the forcing sets of $M$ is called the {\em forcing number} of $M$, which is also called the ``innate degree of freedom'' of a Kekul\'e structure in earlier chemical literature by Klein and Randi\'c \cite{Klein}. We may refer to a survey \cite{Che} on this topic.

In view of this, Vuki\v cevi\'c et al. \cite{Vu1,Vu} introduced the concept of {\em global (or total) forcing set} concerning all perfect matchings instead of a particular perfect matching, which is defined as a subset $S$ of $E(G)$ on which there are no two distinct perfect matchings coinciding, i.e., the restriction of the characteristic function of perfect matchings to $S$ is an injection. The minimum possible cardinality of the global forcing sets is called the {\em global forcing number} of $G$. For more about the global forcing number of a graph, the reader is referred to \cite{Cai,Doslic,Sedlar,H. Zhang}.

Combining the ``forcing" and ``global" ideas, Xu et al. \cite{Xu1} proposed the concept of the complete forcing number of  $G$. A {\em complete forcing set} of $G$ is a subset of $E(G)$ to which the restriction of each perfect matching $M$ is a forcing set of $M$. A complete forcing set with the minimum cardinality is called a {\em minimum complete forcing set} of $G$, and its cardinality is called the {\em complete forcing number} of $G$, denoted by $cf(G)$. The complete forcing number of $G$ can give some sort of identification of the minimal amount of information required not only to distinguish all perfect matchings of $G$, but also to specify forcing sets of all perfect matchings of $G$. They established an equivalent condition for a subset of edges  of a graph to be a complete forcing set and gave an expression for the complete forcing number  of catacondensed hexagonal systems. Further, Chan et. al. \cite{Xu3} obtained that the complete forcing number of a catacondensed hexagonal system is equal to the number of hexagons plus the Clar number and developed a linear-time algorithm for computing it. Besides, some certain explicit formulas for the complete forcing numbers of  primitive coronoids, polyphenyl systems and spiro hexagonal systems has been derived \cite{Liu1,Liu2,Xu2}.

%

In this paper, we give some sharp upper and lower bounds for the complete forcing numbers of  hexagonal systems and use them to determine the complete forcing numbers of some types  of pericondensed hexagonal systems. The paper is organized as follows. In section 2, we  showed that the complete forcing number of a hexagonal system is equal to the sum of that of its normal components whenever it contains fixed edges. And further, we present a sufficient condition for an edge set of a hexagonal system to be a complete forcing set in terms of elementary edge-cut cover. As a direct consequence we obtain  an upper bound on the complete forcing numbers of  hexagonal systems. In section 3,  we establish two sharp lower bounds for the complete forcing numbers of normal hexagonal systems. As applications, in the final section, we give some explicit formulas for the complete forcing numbers of parallelogram, regular hexagon- and rectangle-shaped hexagonal systems.

\section{Preliminaries and a sufficient condition}
Let $G$ be a graph with a perfect matching. A subgraph $G_{0}$ of $G$ is said to be {\em nice} if $G-V(G_{0})$ has a perfect matching. Obviously, an even cycle $C$ of $G$ is nice if and only if there is a perfect matching $M$ of $G$ such that $C\cap M$ is a perfect matching of $C$. For an even cycle $C$, each of the two  perfect matchings of $C$ is called a {\em frame} (or a typeset \cite{Xu1}) of $C$, this concept is used in \cite{Abeledo} to present a min-max theorem.

The following result gives a characterization for a complete forcing set of a graph.

\begin{thm}\label{Xu}\cite{Xu1}
Let $G$ be a graph with  a perfect matching. Then $S\subseteq E(G)$ is a complete forcing set of G if and only if, for any nice cycle $C$ of $G$, the intersection of $S$ and each frame of $C$ is nonempty.
\end{thm}

A hexagonal system (HS) is a 2-connected finite plane graph such that every interior face is a regular hexagon, which can be regarded as the carbon skeleton of benzenoid hydrocarbon molecules. An HS is said to be {\em catacondensed} if no three of its hexagons have a vertex in common, and {\em pericondensed} otherwise. For convenience, we always draw an HS $H$ in the plane such that some of its edges are vertical and color the vertices of $H$ by black and white so that the end-vertices of any edge receive different colors. An edge of $H$ is called an {\em peripheral edge} if it belong to the exterior face of $H$ and {\em inner edge} of $H$ otherwise. 

Let $H$ be an HS with a perfect matching.  An edge $e$ of $H$ is called a {\em fixed double edge} if $e$ is contained in all perfect matchings of $H$ and a {\em fixed single edge} if $e$ is not contained in any perfect matching of $H$. Both fixed double edge and fixed single edge are referred to as {\em fixed edge}. $H$ is said to be {\em normal} if  $H$ has no fixed edge.

\begin{lem} \cite{F. Zhang}\label{normal}
Let $H$ be an HS. Then $H$ is normal if and only if each facial cycle of $H$ is a nice cycle of $H$.
\end{lem}

The non-fixed edges of $H$ form a subgraph of $H$, each component of which is a normal HS \cite{Hansen3} and is called a {\em normal component} of $H$. The complete forcing number of $H$ has the following property.


\begin{thm}\label{GHS}
Let $H$ be an HS with the normal components $H_{1}, H_{2},\ldots,H_{k}$, $k\geq1$. Then
$$cf(H)=\sum_{i=1}^{k}cf(H_{i}).$$
\end{thm}

\begin{proof}
Let $S_{0}$ be a minimum complete forcing set of $H$ and $S_{i}=S_{0}\cap E(H_{i})$ $(i=1,2,\ldots,k)$. Since each nice cycle $C$ of $H_{i}$ is also a nice cycle of $H$, the intersection of $S_{i}$ and each frame of $C$ is nonempty. Hence $S_{i}$ is a complete forcing set of $H_{i}$ by Theorem \ref{Xu}. Then we have
$$cf(H)=|S_{0}|\geq \sum_{i=1}^{k}|S_{i}|\geq \sum_{i=1}^{k}cf(H_{i}).$$

On the other hand, let $S_{i}$ be a minimum complete forcing set of $H_{i}$ for $i=1,2,\ldots,k$ and $S=\bigcup_{i=1}^{k}S_{i}$. We can see that any nice cycle $C$ of $H$ contains no fixed edge of $H$. Therefore, $C$ must be contained in a normal component $H_{i}$ of $H$. By Theorem \ref{Xu}, the intersection of $S_{i}$ and each frame of $C$ is nonempty, so the intersection of $S$ and each frame of $C$ is also nonempty, that is, $S$ is a complete forcing set of $H$. Hence
$$cf(H)\leq|S|= \sum_{i=1}^{k}|S_{i}|= \sum_{i=1}^{k}cf(H_{i}).$$

Consequently, the complete forcing number of $H$ equals to the sum of that of all its normal components.
\end{proof}

Based the above theorem, it is critical to determine the complete forcing numbers of normal HSs.

Next, we present a sufficient condition for an edge set of $E(H)$ to be a complete forcing set of $H$ in terms of elementary edge cut, and thus we get an upper bound of the complete forcing number of $H$. The concept of elementary edge cut was introduced in \cite{Sachs,F. Zhang1} to show the existence of perfect matchings in HS and plays an important role in  resonance theory of  graph \cite{F. Zhang,H. Zhang2} especially the computation of Clar number of  HSs \cite{Hansen,Hansen1}. Two min-max theorems relating to this concept are established as well \cite{Abeledo,H. Zhang3}.

The {\em dual graph} $H^{*}$ of a given HS $H$ can be construct as follows \cite{Bondy}: Let the vertices of $H^{*}$ be the centers of hexagons of $H$ and one vertex on the exterior face of $H$. If two hexagons have a common edge, then we use a segment crossing such edge as an edge of $H^{*}$ to join the centers of such two hexagons. If a hexagon has a peripheral edge, then we use a curve crossing only such an edge as an edge of $H^{*}$ to join the center of the hexagon and the vertex on exterior face. For  $E_{0}\subseteq E(H)$, we denote by $E_{0}^{*}$ the set of edges of $H^{*}$ corresponding to edges of $E_{0}$. 

If $\{V_{1}(H),V_{2}(H)\}$  is a partition of $V(H)$, the set $D$ of all the edges of $H$ that have one end-vertex in $V_{1}(H)$ and the other in $V_{2}(H)$  is called an  {\em edge cut} of $H$. We call $D$ an {\em elementary edge cut} (e-cut for short) of $H$ if $H-D$ has exactly two components such that all edges of $D$ are incident with black vertices of one component, called the black bank of $D$, and white vertices of the other component, called the white bank of $D$.

From the definition of e-cut, we can also  determine whether an edge set of $H$ is an e-cut by the following Lemma from the view of dual graph.

\begin{lem}\label{cut dual}
An edge set $D$ of an HS $H$ is an e-cut if and only if $D^{*}$ induces a cycle of $H^{*}$ and the end-vertices of edges of $D$ have the same color either  inside or  outside of $D^{*}$.
\end{lem}

\begin{figure}[!htbp]
\begin{center}
\includegraphics[totalheight=1.8cm]{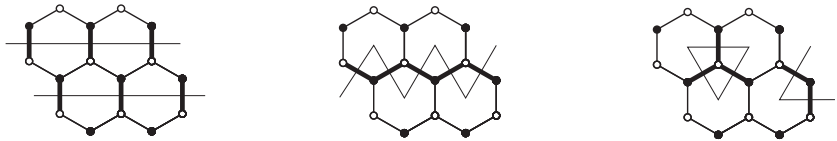}
 \caption{\label{e-cut}\small{An HS with 3 different e-cut covers.}}
\end{center}
\end{figure}

We say that a set of e-cuts $\mathcal{D}=\{D_{1},D_{2},\ldots,D_{k}\}$ {\em covers} $H$ if the boundary of each face (including the exterior face) intersects an e-cut $D_{i}\in \mathcal{D}$ $(1\leq i\leq k)$. In this case, we also call $\mathcal{D}$ or $D=\bigcup_{i=1}^{k}D_{i}$ an {\em e-cut cover} of $H$. For example, Fig. \ref{e-cut} presents three distinct e-cut covers of an HS, where the bold edges denote the edges of e-cut covers of $H$ and the thin lines indicate the cycles of $H^{*}$ corresponding to  the e-cuts covers.

\begin{thm}\label{cut}
Let $H$ be an HS with a perfect matching. If there is a set of e-cuts $\mathcal{D}=\{D_{1},D_{2},\ldots,D_{k}\}$ of $H$ such that any nice cycle $C$ of $H$ intersects an e-cut $D_{i}\in \mathcal{D}$, then $D=\bigcup_{i=1}^{k}D_{i}$ is a complete forcing set of $H$.
\end{thm}

\begin{proof}
Let $C$ be any nice cycle of $H$. Then there is an e-cut $D_{i}\in\mathcal{D}$ such that $C\cap D_{i}\neq\emptyset$.  Let $e_{1}$ be a common edge of $C$ and $D_{i}$. Given an orientation of $C$ along which  $C$ passes through $e_{1}$ from black to white end-vertex. Then  $C$ must return to the  black bank of $D_{i}$ from white bank through another edge $e_{2}$ of $D_{i}$. Hence both $D$ and $C$ have two edges $e_{1}$ and $e_{2}$ in different frames of $C$, which implies that $D$ is a complete forcing set of $H$ by Theorem \ref{Xu}.
\end{proof}

\begin{remark}If $H$ is a normal HS, then the set of e-cuts in Theorem \ref{cut} should be an e-cut cover of $H$ by Lemma \ref{normal}.
\end{remark}

For example, by Theorem \ref{cut}, we can see that three sets consisting of the bold edges of the HS as shown in Fig. \ref{e-cut} are  complete forcing sets.

In particular, all parallel edges of an HS  in any one of three edge directions form an e-cut cover that satisfies the condition of Theorem \ref{cut}. So we have

\begin{cor}
Let $H$ be an HS and $S$ be the set consisting of all parallel edges with the minimum cardinality among three edge directions. Then $cf(H)\leq |S|$.
\end{cor}

We note that the above upper bound can be attained by a linear hexagonal chain (see Theorem \ref{parallelogram}).

\section{Two lower bounds of the complete forcing numbers of normal HSs}
In this section, we establish two lower bounds of the complete forcing numbers of normal HSs.

\begin{thm}\label{low1}
Let $H$ be a normal HS with $n$ hexagons. Then $cf(H)\geq n+1$.
\end{thm}

\begin{proof}
For a face $f$ of $H$ (the exterior face is allowed), let $T_{1}(f)$ and $T_{2}(f)$ denote the two frames of  the boundary of $f$. Let $S$ be a minimum complete forcing set of $H$. Since each facial cycle of $H$ is nice by Lemma \ref{normal}, combining with Theorem \ref{Xu}, we have
$$|S\cap T_{i}(f)|\geq 1,~ i=1 , 2, \text{ for each face $f$ of $H$}.$$
Summing all the above inequalities together, we have
$$2|S|=\sum_{f}(|S\cap T_{1}(f)|+|S\cap T_{2}(f)|)\geq 2(n+1),$$
because each edge of $S$ belongs to exactly two faces of $H$. Then we have
$$cf(H)=|S|\geq n+1.$$
\end{proof}

In the next section, we will show some types of  HSs whose complete forcing numbers attain the above lower bound.

\begin{figure}[!htbp]
\begin{center}
\includegraphics[totalheight=2.0cm]{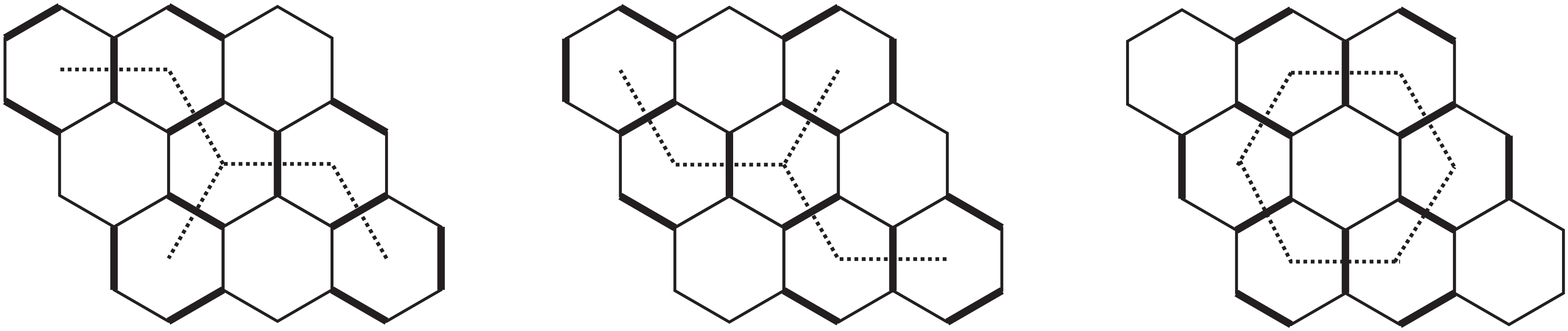}
 \caption{\label{001}\small{The partition of edges  of an HS and the corresponding  dual subgraphs.}}
\end{center}
\end{figure}

\begin{figure}[!htbp]
\begin{center}
\includegraphics[totalheight=3.0cm]{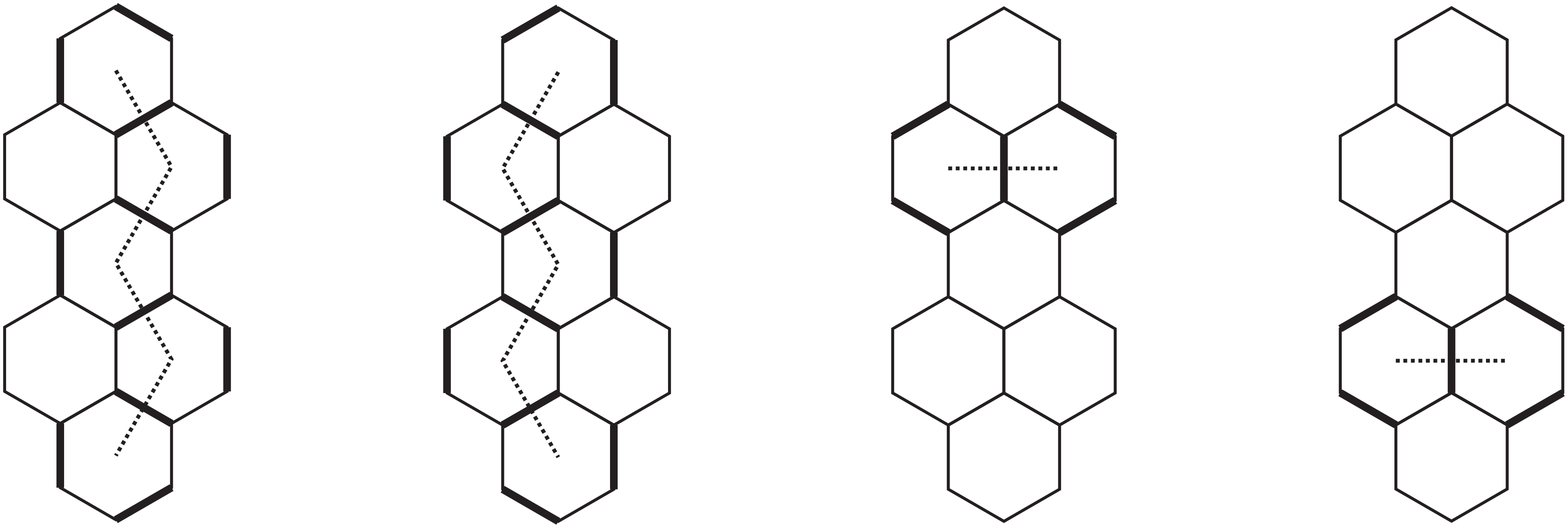}
 \caption{\label{002}\small{The partition of edges  of an HS and the corresponding  dual subgraphs.}}
\end{center}
\end{figure}

To establish another lower bound of the complete forcing number of a normal HS $H$, we partition $E(H)$ into several classes: For $e',e''\in E(H)$, $e'$ and $e''$ belong to the same class if there is a sequence of hexagons $h_{1},h_{2},\ldots,h_{t}$ and a sequence of edges $e_{1},e_{2},\ldots,e_{t-1}$ of $H$ such that $e_{i-1}$ and $e_{i}$ belong to the same frame of $h_{i}$ for $i=1,2,\ldots,t$ where $e_{0}=e'$ and $e_{t}=e''$, and $e'$ and $e''$ belong to  different classes otherwise. Let $E_{1},E_{2},\ldots,E_{k}$  be all classes of edges of $H$ by this partition. It is not difficult to see that two frames of every hexagon belong to different classes and $k\geq 2$. Let $\mathcal{H}_{i}$ be the set of hexagons of $H$ that contain a frame in $E_{i}$  and $H_{i}^{\sharp}$ be the subgraph of dual graph $H^{*}$ induced by the vertices corresponding to the hexagons of $\mathcal{H}_{i}$ ($i=1,2,\ldots,k$). We can see that each hexagon of $H$  belongs to exactly two of $\mathcal{H}_{1},\mathcal{H}_{2},\ldots,\mathcal{H}_{k}$ and $E(H_{i}^{\sharp})\subset E_{i}^{*}$. Figs. \ref{001} and \ref{002} give two examples, where the bold edges indicate the the partition of edges  of given HSs and the dashed edges represent the corresponding  dual subgraphs.

\begin{remark}
Each  $H_{i}^{\sharp}$  is a connected subgraph of a new HS and each inner face of $H_{i}^{\sharp}$ is a hexagon.
\end{remark}

For a graph $G$, we call an edge set of $E(G)$ an {\em edge cover} of $G$ if every vertex of $G$ is incident to some edge of it. The cardinality of a  minimum edge cover of $G$ is called the {\em edge cover number} of $G$, denoted $\rho(G)$. The cardinality of a  maximum matching of $G$ is called the {\em matching number} of $G$, denoted $\nu(G)$ \cite{Lovaz}.

\begin{thm}(Gallai)\label{Gallai}
Let $G$ be a graph without isolated vertices. Then
$$\nu(G)+\rho(G)=|V(G)|.$$
\end{thm}

\begin{thm}\label{lower2}
Let $H$ be a normal HS with $n$ hexagons.  Then
\begin{equation}\label{equ}
cf(H)\geq2n-\sum_{i=1}^{k} \nu(H_{i}^{\sharp}).
\end{equation}
\end{thm}

\begin{proof}
Let $S_{0}$ be a minimum complete forcing set of $H$. Then the intersection of any hexagon $h$ of $\mathcal{H}_{i}$ and $E_{i}$ is one frame of $h$ which contains at least one edge of $S_{0}$ by Theorem \ref{Xu} and Lemma \ref{normal}. Let $S_{i}$ $(i=1,2,\ldots,k)$ be a subset of $E_{i}$ with minimum possible cardinality that contains at least one edge of each hexagon of $\mathcal{H}_{i}$. So $|S_{0}\cap E_{i}|\geq|S_{i}|$, and
\begin{equation}\label{equ1}
cf(H)=|S_{0}|=\sum_{i=1}^{k} |S_{0}\cap E_{i}|\geq \sum_{i=1}^{k} |S_{i}|.
\end{equation}

We claim that
\begin{equation}\label{equ2}
|S_{i}|=|V(H_{i}^{\sharp})|-\nu(H_{i}^{\sharp}), \text{ for $i=1,2,\ldots,k$}.
\end{equation}

If $\mathcal{H}_{i}$ consists of only one hexagon, then $H_{i}^{\sharp}$ is an isolated vertex and the result is trivial.

From now on suppose that $\mathcal{H}_{i}$  consists of at least two hexagons. Then $H_{i}^{\sharp}$ has at least two vertices. We will prove that

\begin{equation}\label{equ4}
|S_{i}|=\rho(H_{i}^{\sharp})=|V(H_{i}^{\sharp})|-\nu(H_{i}^{\sharp}).
\end{equation}

By the definition of $S_{i}$, $S_{i}^{*}$ is a smallest subset of $E_{i}^{*}$ such that every vertex of $H_{i}^{\sharp}$ is incident with at least one edge of $S_{i}^{*}$ (i.e., $S_{i}^{*}$ covers $V(H_{i}^{\sharp})$). Immediately, $|S_{i}^{*}|\leq \rho(H_{i}^{\sharp})$. On the other hand, if there is an edge $e_{1}^{*}\in S_{i}^{*}\setminus E(H_{i}^{\sharp})$ that is incident with a vertex $v^{*}\in V(H_{i}^{\sharp})$, then $v^{*}$ will be incident with another edge $e_{2}^{*}\in E(H_{i}^{\sharp})$ since $H_{i}^{\sharp}$ is is connected and has at least two vertices. Hence we can obtain a subset $(S_{i}^{*}\setminus \{e_{1}^{*}\})\cup \{e_{2}^{*}\}$ of $E_{i}^{*}$ which covers $V(H_{i}^{\sharp})$. This way we replace all edges of $S_{i}^{*}\setminus E(H_{i}^{\sharp})$ with edges of $E(H_{i}^{\sharp})$ and obtain an edge cover $S_{i}'$ of $H_{i}^{\sharp}$ from $S_{i}^{*}$ such that $|S_{i}^{*}|\geq |S_{i}'|\geq \rho(H_{i}^{\sharp})$. Therefore,
$|S_{i}|=|S_{i}^{*}|=\rho(H_{i}^{\sharp})$.
By  Theorem \ref{Gallai}, equation (\ref{equ4}) holds.

Since each hexagon of $H$  belongs to exactly two hexagon sets of $\mathcal{H}_{1},\mathcal{H}_{2},\ldots,\mathcal{H}_{k}$,
\begin{equation}\label{equ3}
\sum_{i=1}^{k}|V(H_{i}^{\sharp})|=2n,
\end{equation}
and by  (\ref{equ1}), (\ref{equ2}) and (\ref{equ3}), we have
$$cf(H)\geq\sum_{i=1}^{k} |S_{i}|=\sum_{i=1}^{k}[|V(H_{i}^{\sharp})|-\nu(H_{i}^{\sharp})]=2n-\sum_{i=1}^{k}\nu(H_{i}^{\sharp}).$$
\end{proof}

\begin{remark}
From the proof of Theorem \ref{lower2}, if $\mathcal{H}_{i}$ consists of only one  hexagon, then  $H_{i}^{\sharp}$ is an isolated vertex and  $|S_{i}|=1=|V(H_{i}^{\sharp})|-\nu(H_{i}^{\sharp})$, but $H_{i}^{\sharp}$ has no edge cover. That is why we use the matching numbers of the dual subgraphs  instead of the edge cover numbers to give the lower bound of the complete forcing number of $H$.
\end{remark}

As a direct application, we show that the complete forcing numbers of all catacondensed HSs attain the lower bound presented in Theorem \ref{lower2}, and thus obtain an alternative method to compute the  complete forcing number of a catacondensed HS.

\begin{lem}\label{cata}\cite{Xu1}
Let $H$ be a catacondensed HS. Then $S\subseteq E(H)$ is a complete forcing set of $H$ if and only if $S$ intersects each frame of every hexagon in $H$.
\end{lem}

\begin{thm}
Let $H$ be a catacondensed HS. Then equality in Ineq. (\ref{equ}) holds.
\end{thm}

\begin{proof}
From the proof of Theorem \ref{lower2}, $S=\bigcup_{i=1}^{k}S_{i}$ is a complete forcing set by Lemma \ref{cata}. Then we have
$$|S|=\sum_{i=1}^{k} |S_{i}|=2n-\sum_{i=1}^{k}\nu(H_{i}^{\sharp})\geq cf(H).$$
Combining with Ineq. (\ref{equ}), equality holds.
\end{proof}

\section{Applications}
In this section, we derive some explicit formulas for the complete forcing numbers of some types of classes of HSs including parallelogram, regular hexagon- and rectangle-shaped HS \cite{cyvin}. Let $n$ be the number of hexagons of the corresponding HS. The main idea is that for a given HS $H$, we will construct a complete forcing set whose cardinality attains the lower bound of the complete forcing number of $H$ by Theorem \ref{low1} or Theorem \ref{lower2}.

For convenience, we denote by $h_{i,j}$ the hexagon of the given HS $H$ in the $i$-th row and the $j$-th column of $H$ from bottle to top and from left to right. Moreover, for a hexagon $h$ of $H$ we denote by $e_{l}(h)$, $e_{tl}(h)$, $e_{tr}(h)$, $e_{r}(h)$, $e_{br}(h)$ and $e_{bl}(h)$ the left vertical edge, the top left edge, the top right edge, the right vertical edge, the bottom right edge and the bottom left edge of $h$ respectively.

\vspace{2mm}

\noindent\textbf{(1) Parallelogram}
\begin{thm}\label{parallelogram}
Let $P(p,q)$ be a parallelogram with with $p$ rows and $q$ columns of hexagons. Then $cf(P(p,q))=pq+1=n+1.$
\end{thm}

\begin{proof}
At first, we have $cf(P(p,q))\geq pq+1=n+1$ by Theorem \ref{low1}. In the following, we construct a complete forcing set $S$ of $P(p,q)$ such that $|S|=pq+1$.

When $p=0$ (mod $3$), we choose $S$ to be the set consisting of the following edges:
the common edges of the $(3i+1)$-th row and the $(3i+2)$-th row of $P(p,q)$, $e_{br}(h_{3i+1,q})$,$e_{tl}(h_{3i+2,1})$, and the inner vertical edges of the $(3i+3)$-th row of $P(p,q)$ for $i=0,1,\ldots,\frac{p-3}{3}$ (see Fig. \ref{para} (1)).

When $p=1$ (mod $3$),  we choose $S$ to be the set consisting of the following edges:
the inner vertical edges of $(3i+1)$-th row of $P(p,q)$, $e_{br}(h_{3i+1,1})$, $e_{tl}(h_{3i+1,q})$, for $i=0,1,\ldots,\frac{p-1}{3}$ and the common edges of the $(3i+2)$-th row and the $(3i+3)$-th row of $P(p,q)$ for $i=0,1,\ldots,\frac{p-4}{3}$ (see Fig. \ref{para} (2)).

When $p=2$ (mod $3$), we choose $S$ to be the set consisting of the following edges:
the common edges of the $(3i+1)$-th row and the $(3i+2)$-th row of $P_{p,q}$, $e_{br}(h_{3i+1,q})$,$e_{tl}(h_{3i+2,1})$ for $i=0,1,\ldots,\frac{p-2}{3}$, and the inner vertical edges of the $(3i+3)$-th row of $P(p,q)$ for $i=0,1,\ldots,\frac{p-5}{3}$ (see Fig. \ref{para} (3)).

\begin{figure}[!htbp]
\begin{center}
\includegraphics[totalheight=5cm]{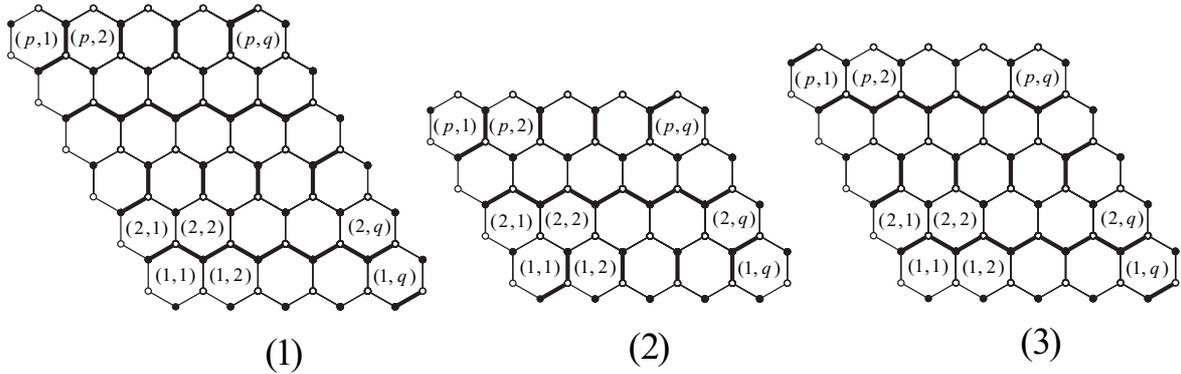}
 \caption{\label{para}\small{Three complete forcing sets of  $P(p,q)$: (1) $p=0$ (mod $3$); (2) $p=1$ (mod $3$); (3) $p=2$ (mod $3$).}}
\end{center}
\end{figure}

In each of the above three cases, we can see that $S$ is an e-cut of $P(p,q)$ which covers $P(p,q)$ by Lemma \ref{cut dual} and any cycle of $P(p,q)$ intersects $S$. Hence $S$ is a complete forcing set of $P(p,q)$ by Theorem \ref{cut}.  Besides, since each edge of $S$ belongs to exactly two faces of $P(p,q)$ and the boundary of each face of $P(p,q)$ has two edges of $S$, we have $2|S|=2(n+1)$, and then $|S|=pq+1=n+1$. Hence $S$ is a minimum complete forcing set of $P(p,q)$ and $cf(P(p,q))=pq+1=n+1$.
\end{proof}

\noindent\textbf{(2) Regular hexagon-shaped HS}
\begin{thm}
Let $H(p)$ be a regular hexagon-shaped HS with $p$ hexagons on each side. Then

$$ cf(H(p))=\left\{
\begin{array}{rcl}
n+1,       &    & {{\rm if}~ p=0~{\rm or}~1~({\rm mod}~3),}\\
n+2,      &      & {\rm otherwise.}\\
\end{array} \right. $$
\end{thm}

\begin{proof}
We divide our proof into the following three cases.

(a) When $p=0$ (mod $3$), we choose $S$ to be the set consisting of the following edges:
the common edges of the $(3i+1)$-th row and the $(3i+2)$-th row of $H(p)$, the inner vertical edges of the $(3i+3)$-th row of $H(p)$, $e_{br}(h_{3i+3,1})$, $e_{bl}(h_{3i+3,p+3i+2})$, the common edges of the $(p+3i+1)$-th row and the $(p+3i+2)$-th row of $H(p)$ for $i=0,1,\ldots,\frac{p-3}{3}$, $e_{tl}(h_{p+1,1})$ and $e_{tr}(h_{p+1,2p-2})$, the inner vertical edges of the $(p+3i+3)$-th row of $H(p)$, $e_{tr}(h_{p+3i+3,1})$, $e_{tl}(h_{p+3i+3,2p-3i-4})$ for $i=0,1,\ldots,\frac{p-6}{3}$ (see Fig. \ref{hexagon} (1)).

We can see that $S$ is an e-cut cover of $H(p)$ which consists of $\frac{2p}{3}$ e-cuts of $H(p)$ by Lemma \ref{cut dual}. Let $C_{1i}$ $(i=0,1,\ldots,\frac{p-3}{3})$ be the cycle bounding the subsystem of $H(p)$ composed of the $(3i+1)-$th row, the $(3i+2)$-th row and the $(3i+3)$-th row of $H(p)$ and $C_{2i}$ $(i=0,1,\ldots,\frac{p-6}{3})$ be the cycle bounding the subsystem of $H(p)$ composed of the $(p+3i+3)$-th row, the $(p+3i+4)$-th row and the $(p+3i+5)$-th row of $H(p)$. We can see that every cycle of $H-S$ can be represented by the symmetric difference of some cycles among $C_{1i}$ $(i=1,2,\ldots,\frac{p}{3})$ or some cycles among $C_{2i}$  $(i=1,\ldots,\frac{p-3}{3})$, which is not a nice cycle of $H(p)$. Hence any other cycle of $H(p)$ intersects $S$ and $S$ is a complete forcing set of $H(p)$ by Theorem \ref{cut}. Besides, since each edge of $S$ belongs to exactly two faces of $H(p)$ and the boundary of each face of $H(p)$ has two edges of $S$, we have $|S|=n+1$ which means that $S$ is a minimum complete forcing set of $H(p)$ by Theorem \ref{low1} and $cf(H(p))=n+1$.

\begin{figure}[!htbp]
\begin{center}
\includegraphics[totalheight=6.9cm]{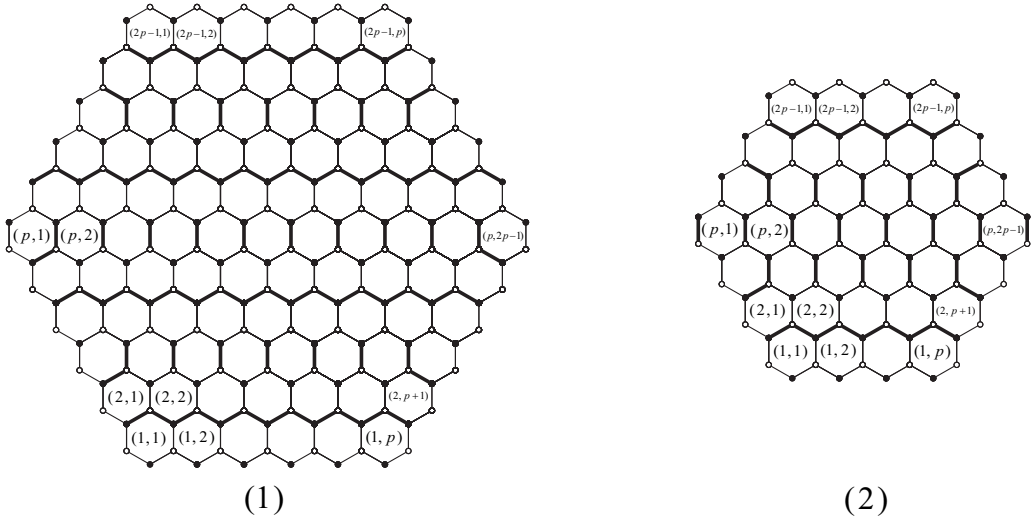}
 \caption{\label{hexagon}\small{Two complete forcing sets of $H(p)$: (1) $p=0$ (mod $3$), (2) $p=1$ (mod $3$).}}
\end{center}
\end{figure}

(b) When $p=1$ (mod $3$), if $p=1$, then the result is trivial, so we suppose that $p\neq1$. We choose $S$ to be the set consisting of the following edges:
the common edges of the $(3i+1)$-th row and the $(3i+2)$-th row of $H(p)$, the inner vertical edges of the $(3i+3)$-th row of $H(p)$, $e_{br}(h_{3i+3,1})$, $e_{bl}(h_{3i+3,p+3i+2})$, the inner vertical edges of the $(p+3i+1)$-th row of $H(p)$, $e_{tr}(h_{p+3i+1,1})$, $e_{tl}(h_{p+3i+1,2p-3i-2})$, the common edges of the $(p+3i+2)$-th row and the $(p+3i+3)$-th row of $H(p)$ for $i=0,1,\ldots,\frac{p-4}{3}$ and all the vertical edges of the $p$-th row of $H(p)$ (see Fig. \ref{hexagon} (2)).

We can see that $S$ is an e-cut cover of $H(p)$ which consists of $\frac{2p+1}{3}$ e-cuts by Lemma \ref{cut dual}. Let $C_{1i}$ $(i=0,1,\ldots,\frac{p-4}{3})$ be the cycle bounding the subsystems of $H(p)$ composed of the $(3i+1)$-th row, the $(3i+2)$-th row and the $(3i+3)$-th row of $H(p)$ and $C_{2i}$ $(i=0,1,\ldots,\frac{p-4}{3})$ be the cycle bounding the subsystems of $H(p)$ composed of the $(p+3i+1)$-th row, the $(p+3i+2)$-th row and the $(p+3i+3)$-th row of $H(p)$. We can see that every cycle of $H-S$ can be represented by the symmetric difference of some cycles among $C_{1i}$ $(i=0,1,\ldots,\frac{p-4}{3})$ or some cycles among $C_{2i}$ $(i=0,1,\ldots,\frac{p-4}{3})$,  which is not a nice cycle of $H(p)$.  Hence any other cycle  of $H(p)$ intersects $S$ and $S$ is a complete forcing set of $H(p)$ by Theorem \ref{cut}. Besides, since each edge of $S$ belongs to exactly two faces of $H(p)$ and the boundary of each face of $H(p)$ has two edges of $S$, we have $|S|=n+1$ which means that $S$ is a minimum complete forcing set of $H(p)$ by Theorem \ref{low1} and $cf(H(p))=n+1$.

(c) When $p=2$ (mod $3$), we choose $S$ to be the set consisting of the following edges:
the common edges of the $(3i+1)$-th row and the $(3i+2)$-th row of $H(p)$, the inner vertical edges of the $(3i+3)$-th row of $H(p)$, $e_{br}(h_{3i+3,1})$, $e_{bl}(h_{3i+3,p+3i+2})$, the inner vertical edges of the $(p+3i+2)$-th row of $H(p)$, $e_{tr}(h_{p+3i+2,1})$, $e_{tl}(h_{p+3i+2,2p-3i-3})$, the common edges of the $(p+3i+3)$-th row and the $(p+3i+4)$-th row of $H(p)$ for $i=0,1,\ldots,\frac{p-5}{3}$, the common edges of the $(p-1)$-th row and the $p$-th row of $H(p)$, $e_{bl}(h_{p,1})$, $e_{br}(h_{p,2p-1})$, and all vertical edges of the $(p+1)$-th row of $H(p)$ (see Fig. \ref{hexagon2}).

\begin{figure}[!htbp]
\begin{center}
\includegraphics[totalheight=6.0cm]{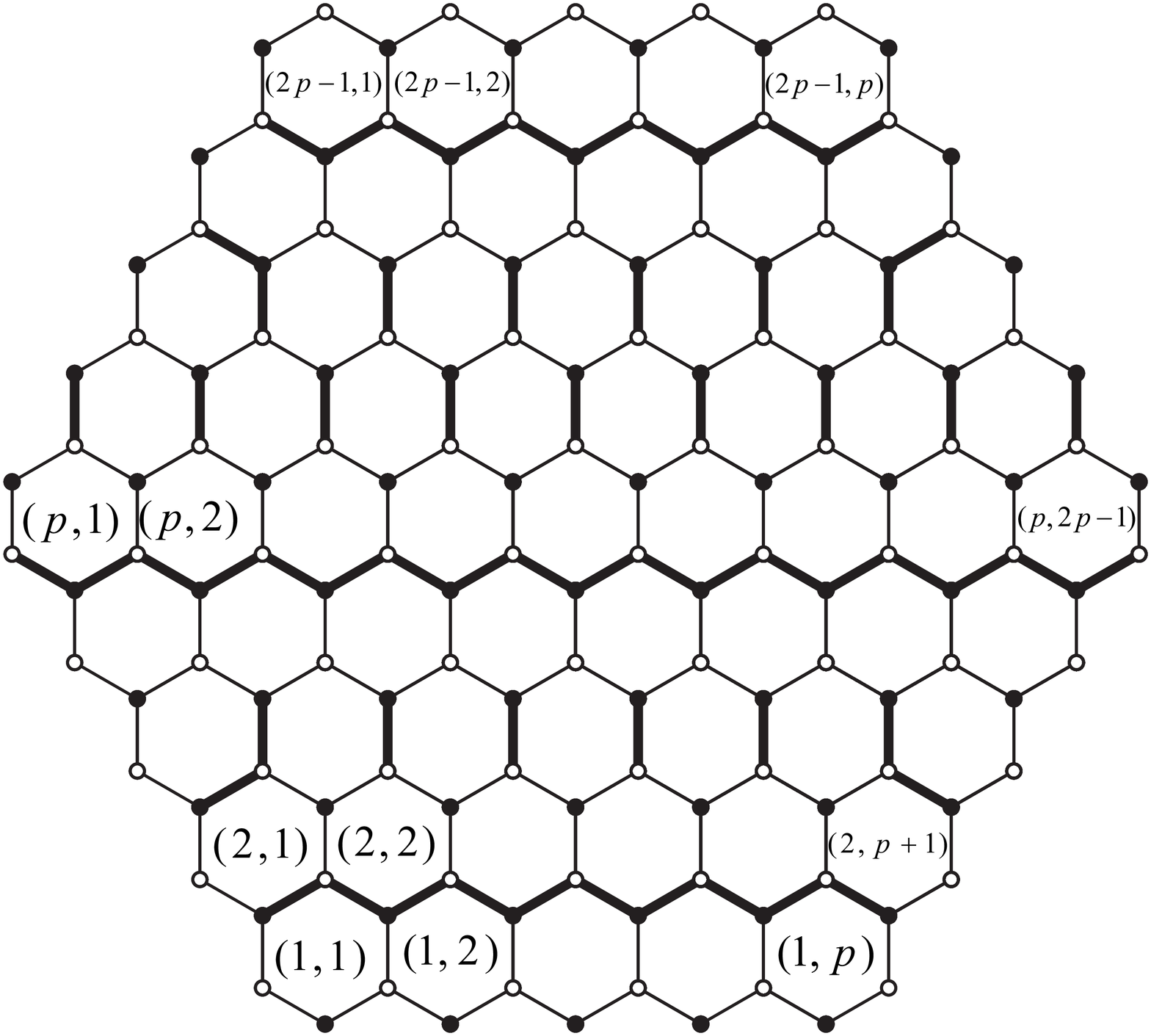}
 \caption{\label{hexagon2}\small{A complete forcing set of $H(p)$  when $p=2$ (mod $3$).}}
\end{center}
\end{figure}

We can see that $S$ is an e-cut cover of $H(p)$ which consists of $\frac{2p+2}{3}$ e-cuts of $H(p)$  by Lemma \ref{cut dual}. Similar to the discussions in the above 2 cases, we can show that $S$ is a complete forcing set of $H(p)$ by Theorem \ref{cut}. Since each edge of $S$ belongs to exactly two faces of $H(p)$, the boundary of each inner face of $H(p)$ has 2 edges of $S$ and  the boundary of the exterior face has 4 edges of $S$, we have $2|S|=2n+4$, and then $|S|=n+2$.

\begin{figure}[!htbp]
\begin{center}
\includegraphics[totalheight=4.9cm]{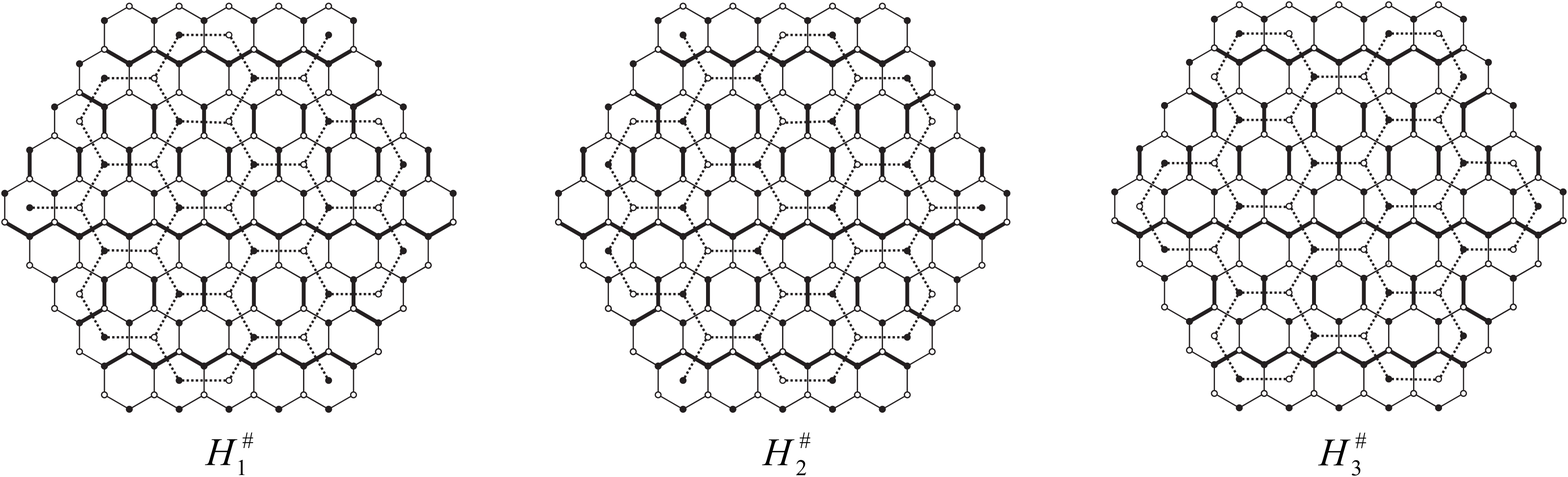}
 \caption{\label{hexagon3}\small{The three corresponding dual subgraphs $H_{1}^{\sharp},H_{2}^{\sharp}$ and $H_{3}^{\sharp}$.}}
\end{center}
\end{figure}

On the other hand, the edges of $H(p)$ can be partitioned into three classes such that $\mathcal{H}_{1}$ contains the hexagons $h_{p,1}$ and $h_{p+1,2p-2}$ of $H(p)$, $\mathcal{H}_{2}$ contains the hexagons $h_{p,2p-1}$ and $h_{p+1,1}$ of $H(p)$, $\mathcal{H}_{3}$ contains the hexagons $h_{p,1}$, $h_{p,2p-1}$, $h_{p+1,1}$ and $h_{p+1,2p-2}$ of $H(p)$. By our construction of $S$, each of the four hexagons $h_{p,1}$, $h_{p,2p-1}$, $h_{p+1,1}$ and $h_{p+1,2p-2}$  has one edge of $S$ that intersects an edge of $E(H_{3}^{\sharp})$ and has no edge of $S$ that intersects an edge of  $E(H_{i}^{\sharp})$ $(i=1,2)$. Except the above four hexagons, any other hexagon of $\mathcal{H}_{i}$ has one edge of $S$ that intersects an edge of  $E(H_{i}^{\sharp})$ for $i=1,2$ or $3$, because each frame of such hexagon has one inner edge of $H(p)$ that belongs to $S$. Hence the edges of $H_{i}^{\sharp}$ that intersect an edge of $S$ are a matching of $H_{i}^{\sharp}$ for $i=1,2$ or $3$. Moreover, the edges of $H_{1}^{\sharp}$ and $H_{2}^{\sharp}$ that intersect an edge of  $S$ are a maximum matching of $H_{1}^{\sharp}$ and $H_{2}^{\sharp}$ respectively, because all  uncovered vertices of $H_{1}^{\sharp}$ are the vertices corresponding to the hexagons $h_{p,1}$ and $h_{p+1,2p-2}$ of $\mathcal{H}_{1}$ that have the same color; and all uncovered vertices of $H_{2}^{\sharp}$ are the vertices corresponding to the hexagons $h_{p,2p-1}$ and $h_{p+1,1}$ of $\mathcal{H}_{2}$ that have the same color. Besides, we can see that the edges of $H_{3}^{\sharp}$ that intersect an edge of  $S$ are a perfect matching of $H_{3}^{\sharp}$ (see Fig. \ref{hexagon3}). Since each edge of $S$ that is an inner edge of $H(p)$ intersects exactly one edge of $E(H_{1}^{\sharp})\cup E(H_{2}^{\sharp})\cup E(H_{3}^{\sharp})$  and $S$  has 4 peripheral edges of $H(p)$, the number of edges of $S$ that are the inner edges of $H(p)$ is $\sum_{i=1}^{3} \nu(H_{i}^{\sharp})=|S|-4=n+2-4=n-2$. By Theorem \ref{lower2}, $cf(H)\geq2n-\sum_{i=1}^{k} \nu(H_{i}^{\sharp})=2n-(n-2)=n+2$. Therefore, $S$ is a minimum complete forcing set of $H(p)$ and $cf(H(p))=n+2$.
\end{proof}

\noindent\textbf{(3) Rectangle-shaped HS}

\begin{figure}[!htbp]
\begin{center}
\includegraphics[totalheight=4.7cm]{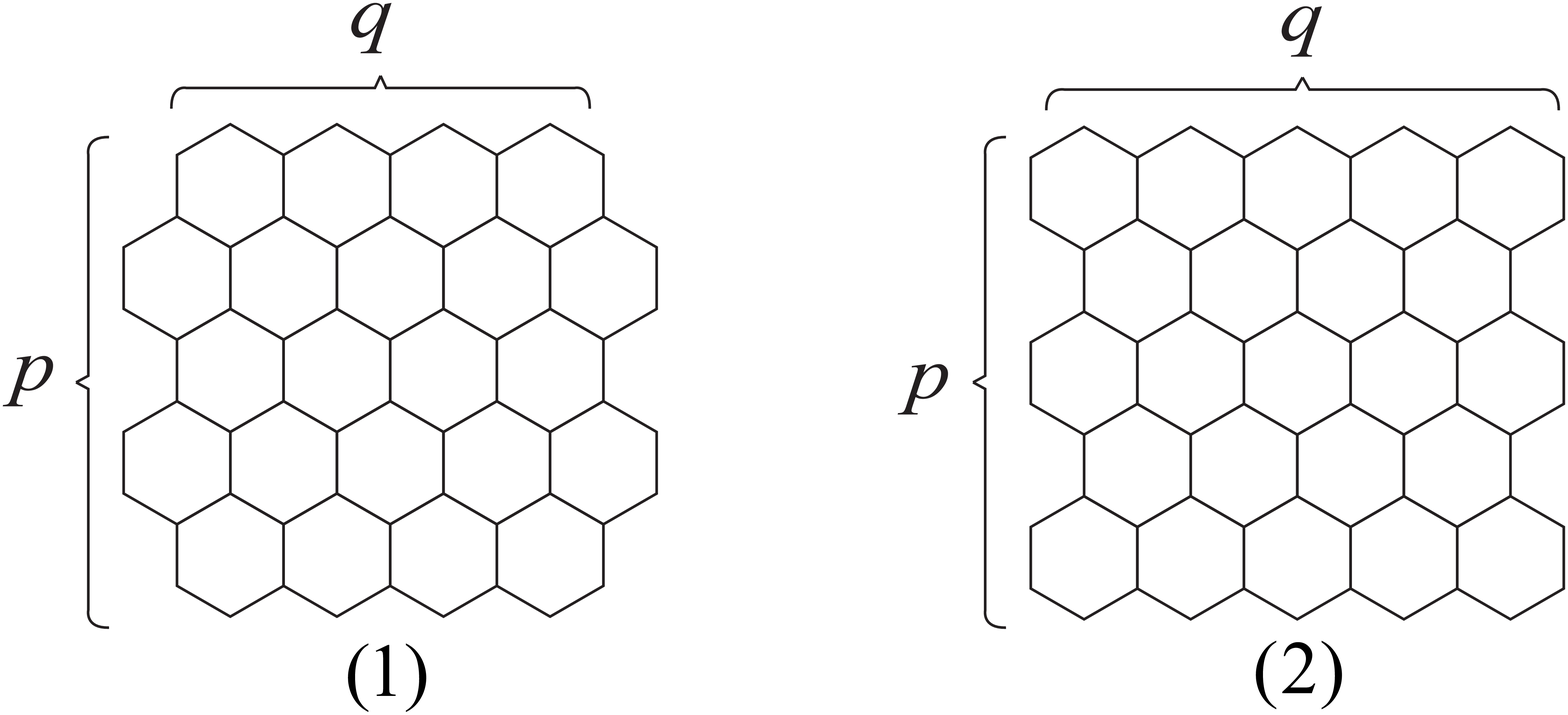}
 \caption{\label{rect0}\small{(1) A oblate rectangle-shaped HS; (2) a prolate rectangle-shaped HS.}}
\end{center}
\end{figure}

\begin{thm}
Let $R_{o}(p,q)$ be a oblate rectangle-shaped HS with parameter $p$ and $q$ as shown in Fig. \ref{rect0} (1). Then
$$ cf(R_{o}(p,q))=\left\{
\begin{array}{rcl}
n+1,       &    & {{\rm if}~ q=1~({\rm mod}~3),}\\
n+\frac{p+1}{2},      &      & {\rm otherwise.}\\
\end{array} \right. $$
\end{thm}

\begin{proof}
We divide our proof into the following three cases.

\begin{figure}[!htbp]
\begin{center}
\includegraphics[totalheight=4.0cm]{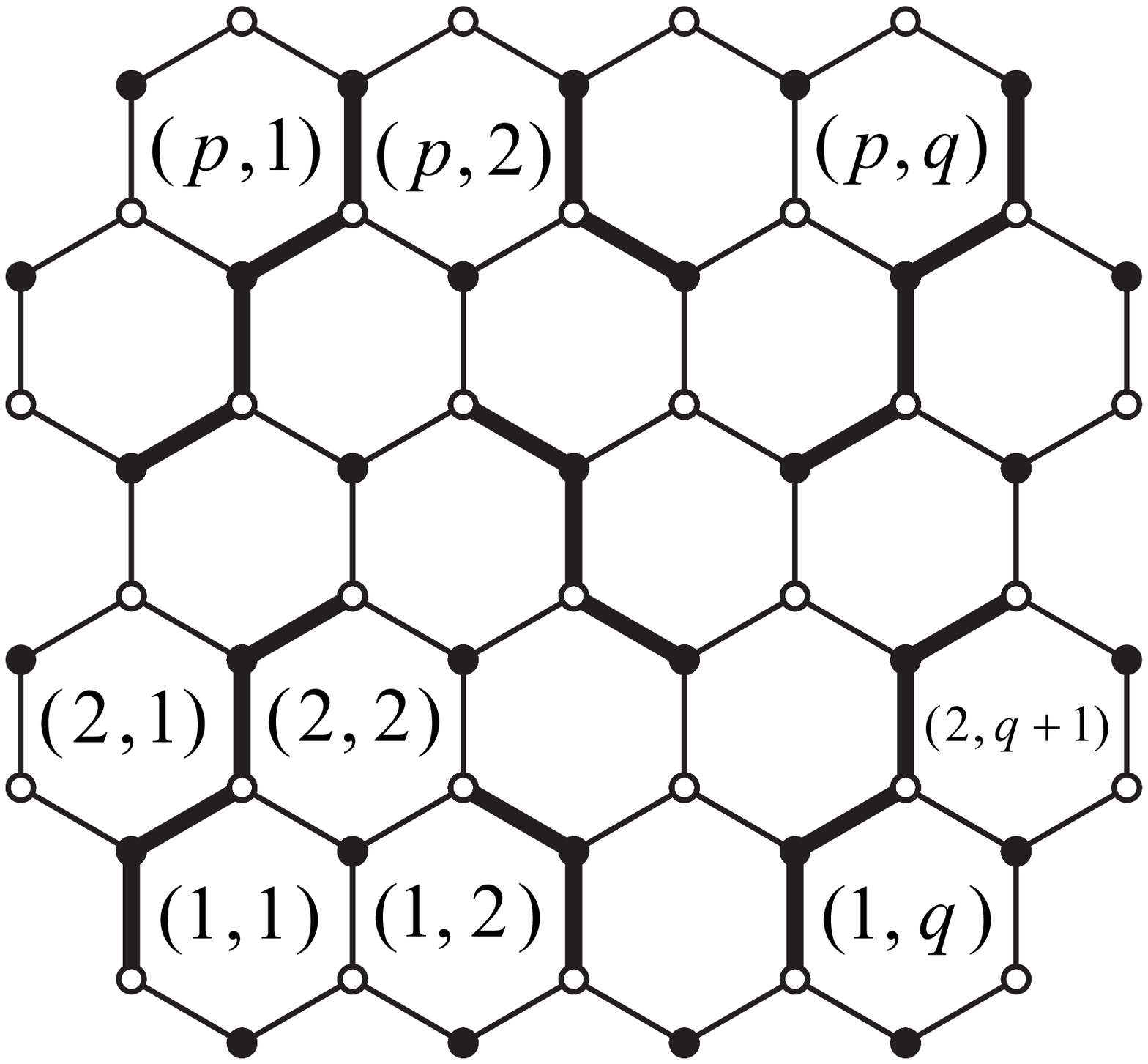}
 \caption{\label{rect}\small{A complete forcing set of $R_{o}(p,q)$ when $q=1$ (mod $3$).}}
\end{center}
\end{figure}

(a) When $q=1$ (mod $3$), we choose $S$ to be the set consisting of the following edges:
$e_{l}(h_{1,3j+1})$, $e_{br}(h_{2i,3j+1})$, $e_{l}(h_{2i,3j+2})$, $e_{tl}(h_{2i,3j+2})$, $e_{r}(h_{p,3j+1})$ for $i=1,2,\ldots,\frac{p-1}{2}, j=0,1,2,\ldots,\frac{q-1}{3}$, $e_{bl}(h_{2i,3j})$, $e_{tr}(h_{2i,3j})$ for $i=1,2,\ldots,\frac{p-1}{2}, j=1,2,\ldots,\frac{q-1}{3}$, and $e_{r}(h_{2i-1,3j-1})$  for $i=1,2,\ldots,\frac{p+1}{2}, j=1,2,\ldots,\frac{q-1}{3}$ (see Fig. \ref{rect}).

We can see that $S$ is an e-cut of $R_{o}(p,q)$ and covers $R_{o}(p,q)$  by Lemma \ref{cut dual} and any cycle of $R_{o}(p,q)$ intersects $S$. Hence $S$ is a complete forcing set of $R_{o}(p,q)$ by Theorem \ref{cut}. On the other hand, since each edge of $S$ belongs to exactly two faces of $R_{o}(p,q)$ and the boundary of each face of $R_{o}(p,q)$ has two edges of $S$, we have $|S|=n+1$. Therefore, by Theorem \ref{low1}, $S$ is a minimum complete forcing set of $R_{o}(p,q)$ and $cf(R_{o}(p,q))=n+1$.


\begin{figure}[!htbp]
\begin{center}
\includegraphics[totalheight=4.0cm]{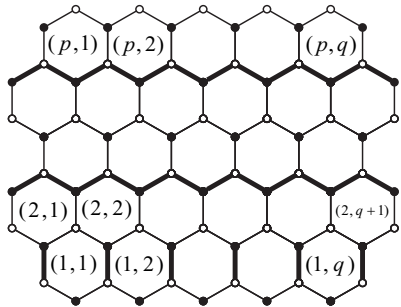}
 \caption{\label{rect2}\small{A complete forcing set of $R_{o}(p,q)$ when $q=2$ (mod $3$).}}
\end{center}
\end{figure}

(b) When $q=2$ (mod $3$), let $S$ be the edge set consisting the following edges:
the common edges of the $(2i)$-th row and the $(2i+1)$-th row of $R_{o}(p,q)$, $e_{tl}(h_{2i,1})$, $e_{tr}(h_{2i,q+1})$ for $i=1,2,\ldots,\frac{p-1}{2}$ and all vertical edges of the first row of $R_{o}(p,q)$ (see Fig. \ref{rect2}). We can see that $S$ is an elementary cut cover of $R_{o}(p,q)$ which consists of $\frac{p+1}{2}$ elementary cuts of $R_{o}(p,q)$  by Lemma \ref{cut dual}, and any cycle of $R_{o}(p,q)$ intersects  $S$. By Theorem \ref{cut}, $S$ is a complete forcing set of $R_{o}(p,q)$. Since each edge of $S$ belongs to exactly two faces of $R_{o}(p,q)$, the boundary of each inner face of $R_{o}(p,q)$ has 2 edges of $S$ and the boundary of the exterior face has $p+1$ edges of $S$, we obtain that $|S|=n+\frac{p+1}{2}$.

\begin{figure}[!htbp]
\begin{center}
\includegraphics[totalheight=4.5cm]{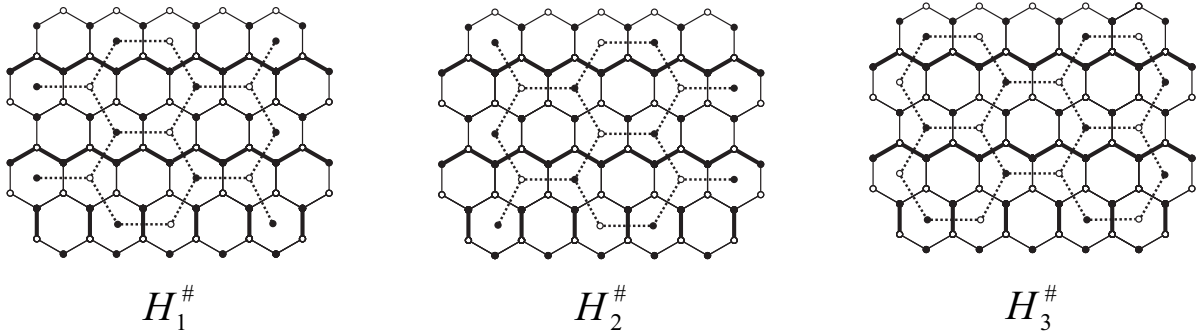}
 \caption{\label{rect3}\small{The three corresponding dual subgraphs $H_{1}^{\sharp},H_{2}^{\sharp}$ and $H_{3}^{\sharp}$.}}
\end{center}
\end{figure}

On the other hand, the edges of $R_{o}(p,q)$ has can be partitioned into three classes such that $\mathcal{H}_{1}$ contains the hexagons $h_{1,q}$ and $h_{2i,1}$ $(i=1,2,\ldots,\frac{p-1}{2})$ of $R_{o}(p,q)$, $\mathcal{H}_{2}$ contains the hexagons $h_{1,1}$ and $h_{2i,q+1}$ $(i=1,2,\ldots,\frac{p-1}{2})$ of $R_{o}(p,q)$, $\mathcal{H}_{3}$ contains the hexagons $h_{1,1}$, $h_{1,q}$, $h_{2i,1}$ and $h_{2i,q+1}$ $(i=1,2,\ldots,\frac{p-1}{2})$ of $R_{o}(p,q)$. By our construction of $S$, each of the $p+1$ hexagons $h_{1,1}$, $h_{1,q}$, $h_{2i,1}$ and $h_{2i,q+1}$ $(i=1,2,\ldots,\frac{p-1}{2})$  has one edge of $S$ that intersects an edge of $H_{3}^{\sharp}$ and has no edge of $S$ that intersects an edge of $H_{i}^{\sharp}$ $(i=1,2)$. Except the above $p+1$ hexagons, any other hexagon of $\mathcal{H}_{i}$ has one edge of $S$ that intersects an edge of $H_{i}^{\sharp}$  for $i=1,2$ or $3$,  because each frame of such hexagon has one inner edge of $R_{o}(p,q)$ that belongs to $S$. Hence the edges of $H_{i}^{\sharp}$ that intersect an edge of $S$ are a matching of $H_{i}^{\sharp}$  for $i=1,2$ or $3$. Moreover, the edges of $H_{1}^{\sharp}$ and $H_{2}^{\sharp}$ that intersect an edge of $S$ are a maximum matching of $H_{1}^{\sharp}$ and $H_{2}^{\sharp}$ respectively, because all the uncovered vertices of $H_{1}^{\sharp}$ are the vertices corresponding to the hexagons $h_{1,q}$ and $h_{2i,1}$ $(i=1,2,\ldots,\frac{p-1}{2})$ of $\mathcal{H}_{1}$ that have the same color;  and all the uncovered vertices of $H_{2}^{\sharp}$ are the vertices corresponding to the hexagons $h_{1,1}$ and $h_{2i,q+1}$ $(i=1,2,\ldots,\frac{p-1}{2})$ of $\mathcal{H}_{2}$ that have the same color. Besides, we can see that the edges of $H_{3}^{\sharp}$ that intersect an edge of $S$ are a perfect matching of $H_{3}^{\sharp}$ (see Fig. \ref{rect3}). Since each edge of $S$ that is an inner edge of $R_{o}(p,q)$ intersects exactly one edge of $E(H_{1}^{\sharp})\cup E(H_{2}^{\sharp})\cup E(H_{3}^{\sharp})$ and $S$  has  $p+1$ peripheral edges of $R_{o}(p,q)$ , the number of edges of $S$ that are the inner edges of $R_{o}(p,q)$ is $\sum_{i=1}^{3} \nu(H_{i}^{\sharp})=|S|-(p+1)=n+\frac{p+1}{2}-(p+1)=n-\frac{p+1}{2}$.   By Theorem \ref{lower2}, $cf(H)\geq2n-\sum_{i=1}^{k} \nu(H_{i}^{\sharp})=2n-(n-\frac{p+1}{2})=n+\frac{p+1}{2}$. Therefore, $S$ is a minimum complete forcing set of $R_{o}(p,q)$ and $cf(R_{o}(p,q))=n+\frac{p+1}{2}$.

\begin{figure}[!htbp]
\begin{center}
\includegraphics[totalheight=4.0cm]{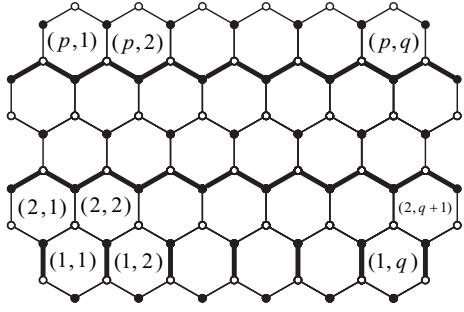}
 \caption{\label{rect4}\small{A complete forcing set of $R_{o}(p,q)$ when $q=0$ (mod $3$).}}
\end{center}
\end{figure}

(c) When $q=0$ (mod $3$), let $S$ be the edge set consisting the following edges:
the common edges of the $(2i)$-th row and the $(2i+1)$-th row of $R_{o}(p,q)$, $e_{tl}(h_{2i,1})$, $e_{tr}(h_{2i,q+1})$ for $i=1,2,\ldots,\frac{p-1}{2}$ and all vertical edges of the first row of $R_{o}(p,q)$. We can see that $S$ is an elementary cut cover of $R_{o}(p,q)$ which consists of $\frac{p+1}{2}$ elementary cuts of $R_{o}(p,q)$  by Lemma \ref{cut dual} and $|S|=n+\frac{p-1}{2}$, and any cycle of $R_{o}(p,q)$ intersect  $S$. By Theorem \ref{cut}, $S$ is a complete forcing set of $R_{o}(p,q)$ (see Fig. \ref{rect4}).

On the other hand, the edges of $R_{o}(p,q)$ can be partitioned into three classes. We can see that the edges of $H_{1}^{\sharp}$, $H_{2}^{\sharp}$ and $H_{3}^{\sharp}$ that intersect an edge of $S$ are a maximum matching of $H_{1}^{\sharp}$, $H_{2}^{\sharp}$ and $H_{3}^{\sharp}$ respectively, because all the uncovered vertices of $H_{1}^{\sharp}$ are the vertices corresponding to the hexagons $h_{2i,1}$ $(i=1,2,\ldots,\frac{p-1}{2})$ of $\mathcal{H}_{1}$ that have the same color; all the uncovered vertices of $H_{2}^{\sharp}$ are the vertices corresponding to the hexagons $h_{2i,q+1}$ $(i=1,2,\ldots,\frac{p-1}{2})$ of $\mathcal{H}_{2}$ that have the same color; and all the uncovered vertices of $H_{3}^{\sharp}$ are the vertices corresponding to the hexagons $h_{1,1}$ and $h_{1,q}$ of $\mathcal{H}_{3}$ and there exists no augmenting path in $H_{3}^{\sharp}$ relative to the matching of $H_{3}^{\sharp}$ (see Fig. \ref{rect5}). Similar as the above case, by Theorem \ref{lower2}, $cf(H)\geq2n-\sum_{i=1}^{k} \nu(H_{i}^{\sharp})=2n-(n-\frac{p+1}{2})=n+\frac{p+1}{2}$.  Hence $S$ is a minimum complete forcing set of $R_{o}(p,q)$ and $cf(R_{o}(p,q))=n+\frac{p+1}{2}$.
\end{proof}
\begin{figure}[!htbp]
\begin{center}
\includegraphics[totalheight=4.0cm]{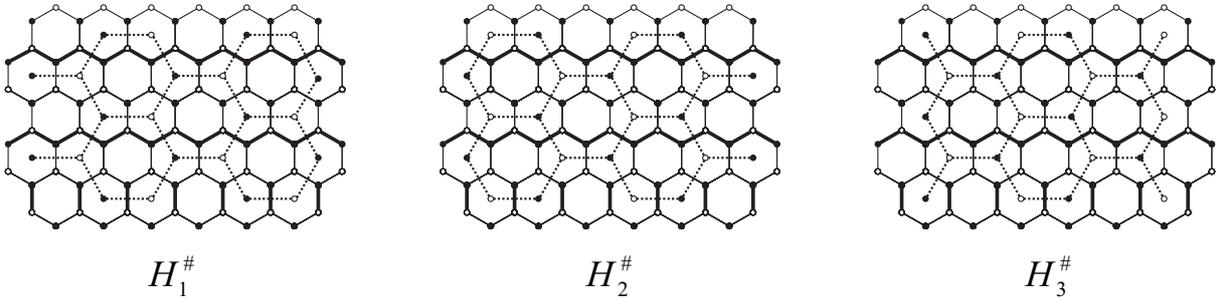}
 \caption{\label{rect5}\small{The three corresponding dual subgraphs $H_{1}^{\sharp},H_{2}^{\sharp}$ and $H_{3}^{\sharp}$.}}
\end{center}
\end{figure}

\begin{thm}
Let $R_{p}(p,q)$ be a prolate rectangle-shaped HS with parameter $p$ and $q$ as shown in Fig. \ref{rect0} (2). Then  $$cf(R_{p}(p,q))=\frac{p+1}{2}(q+1).$$
\end{thm}

\begin{proof}
By deleting all fixed edges of $R_{p}(p,q)$, we obtain $\frac{p+1}{2}$ normal components of $R_{p}(p,q)$ each of which is a linear hexagonal chain $P(1,q)$. By Theorem \ref{GHS} and Theorem \ref{parallelogram}, we have
$$cf(R_{p}(p,q))=\frac{p+1}{2}(q+1).$$
\end{proof}


\begin{thebibliography}{99}


    \bibitem{Abeledo}H. Abeledo, G. W. Atkinson, A min-max theorem for plane bipartite graphs, Discrete Appl. Math. 158 (2010) 375-378.

    \bibitem{Bondy}J. A. Bondy, U.S.R. Murty, Graph Theory with Applications, American Elsevier, New York, Macmillan, London, 1976.

    \bibitem{Che}Z. Che, Z. Chen, Forcing on perfect matchings-A survey, MATCH Commun. Math. Comput. Chem. 66 (2011) 93-136.


    \bibitem{cyvin}S. J. Cyvin, I. Gutman, Kekul\'e Structures in Benzenoid Hydrocarbons, Springer, Berlin, 1988.

    \bibitem{Cai}J. Cai, H. Zhang, Global forcing number of some chemical graphs, MATCH Commun. Math. Comput. Chem. 67 (2012) 289-312.

    \bibitem{Doslic}T. Do\v sli\'c, Global forcing number of benzenoid graphs, J. Math. Chem. 41 (2007) 217-229.

    \bibitem{Xu3}W. Chan, S. Xu, G. Nong, A linear-time algorithm for computing the complete forcing number and the Clar number of catacondensed hexagonal systems, MATCH Commun. Math. Comput. Chem. 74 (2015) 201-216.

    \bibitem{Hansen}P. Hansen, M. Zheng, Upper bounds for the Clar number of benzenoid hydrocarbons, Journal of the Chemical Society, Faraday Transactions 88 (1992) 1621-1625.

    \bibitem{Hansen3}P. Hansen, M. Zheng, Normal components of benzenoid systems, Theorefica Chimica Acta, 85 (1993) 335-344.

    \bibitem{Hansen1}P. Hansen, M. Zheng, The Clar number of a benzenoid hydrocarbon and linear programming, J. Math. Chem. 15 (1994) 93-107.

    \bibitem{Harary}F. Harary, D. J. Klein, T. P. \v Zivkovi\'c, Graphical properties of polyhexes: perfect matching vector and forcing, J. Math. Chem. 6 (1991) 295-306.

    \bibitem{Klein}D. J. Klein, M. Randi\'c, Innate degree of freedom of a graph, J. Comput. Chem. 8 (1987) 516-521.


    \bibitem{Lovaz}L. Lov\'asz, M. D. Plummer, Matching Theory, Annals of Discrete Mathematics, Vol. 29, North-Holland, Amsterdam, 1986.

    \bibitem{Liu2}B. Liu, H. Bian, H. Yu, Complete forcing numbers of polyphenyl systems, Iranian Jounal of Mathematical Chemistry, 7 (2016) 39-46.



    \bibitem{Liu1}B. Liu, H. Bian, H. Yu, J. Li, Complete forcing number of spiro hexagonal systems, Polyc. Arom. Comp. https://doi.org/10.1080/10406638.2019.1600560.

    \bibitem{Mahmoodian}E. S. Mahmoodian, R. Naserasr, M. Zaker, Defining sets in vertex colorings of graphs and Latin rectangles, Discrete Math. 167 (1997) 451-460.


    \bibitem{Sachs}H. Sachs,  Perfect matchings in hexagonal system,  Combinatorica 4 (1984) 89-99.

    \bibitem{Sedlar}J. Sedlar, The global forcing number of the parallelogram polyhex, Discrete Appl. Math. 160 (2012) 2306-2313.


    \bibitem{Vu1}D. Vuki\v cevi\'c, T. Do\v sli\'c, Global forcing number of grid graphs, Austral. J. Combin. 38 (2007) 47-62.

    \bibitem{Vu}D. Vuki\v cevi\'c, J. Sedlar, Total forcing number of the triangular grid, Math, Commun. 9 (2004) 169-179.

    \bibitem{Xu1}S. Xu, H. Zhang, J. Cai, Complete forcing numbers of catacondensed hexagonal systems, J. Comb. Opt. 29 (2015) 803-814.

    \bibitem{Xu2}S. Xu, X. Liu, W. Chan, H, Zhang, Complete forcing numbers of primitive coronoids, J. Comb. Opt. 32 (2016) 318-330.

    \bibitem{F. Zhang}F. Zhang, R. Chen, When each hexagon of a hexagonal system covers it, Discrete Appl. Math. 30 (1991) 63-75.

    \bibitem{F. Zhang1}F. Zhang, R. Chen, X. Guo, Perfect matchings in hexagonal systems, Graphs Combin. 1 (1985) 383-386.

    \bibitem{H. Zhang}H. Zhang, J. Cai, On the global forcing number of hexagonal systems, Discrete Appl. Math. 162 (2014) 334-347.


    \bibitem{H. Zhang3}H. Zhang, H. Yao, D. Yang, A min-max result on outerplane bipartite graphs, Appl. Math. Lett. 20 (2007) 199-205.

    \bibitem{H. Zhang2}H. Zhang, F. Zhang, Plane elementary bipartite graphs, Discrete Appl. Math. 105 (2000) 473-490.




\end{thebibliography}
 \end{document}